\documentclass[11pt]{article}

\usepackage[a4paper,margin=1.02in]{geometry}
\usepackage[T1]{fontenc}
\usepackage{lmodern}
\usepackage{amsmath,amssymb,amsthm,mathtools}
\numberwithin{equation}{section}
\usepackage{microtype}
\usepackage{cite}
\usepackage{xcolor}
\usepackage[hidelinks]{hyperref}

\newtheorem{theorem}{Theorem}[section]
\newtheorem{lemma}[theorem]{Lemma}
\newtheorem{proposition}[theorem]{Proposition}

\newtheorem{claim}[theorem]{Claim}
\newtheorem{corollary}[theorem]{Corollary}
\newtheorem{definition}[theorem]{Definition}
\newtheorem{observation}[theorem]{Observation}

\newtheorem{conjecture}[theorem]{Conjecture}

\newcommand{\F}{\mathbb F}
\newcommand{\Gr}{\operatorname{Gr}}
\newcommand{\supp}{\operatorname{supp}}

\title{The Structure of Cycles in Projective Geometry over $\F_q$}
\author{
Ran J. Tessler\thanks{Weizmann Institute. Email:
\href{mailto:ran.tessler@weizmann.ac.il}{ran.tessler@weizmann.ac.il}.
Supported by the ISF (grant No.~1729/23).}
\and
Elad Tzalik\thanks{Weizmann Institute. Email:
\href{mailto:elad.tzalik@weizmann.ac.il}{elad.tzalik@weizmann.ac.il}.
Supported by the Adams Fellowship Program of
the Israel Academy of Sciences and Humanities.}
}
\date{}

\begin{document}
\maketitle
\vspace{-1.2em}

\begin{abstract}

A classical geometric result says that every nonzero cycle of the mod-$2$ incidence map from $d$-subsets to $(d-1)$-subsets of $[n]$ has support at least $d+1$, with equality attained by the boundary of a simplex on $d+1$ vertices. We prove an analogous result for the subspace lattice of $\F_q^n$, determining the minimum support size of a nonzero $d$-cycle over a field $K$ of characteristic $p \mid q+1$.
Surprisingly, the boundary of a $(d+1)$-space is not always optimal. Shorter cycles occur for $d=1$, and for $d=2$ when $n\ge4$, and otherwise, the boundary of a $(d+1)$-space is shortest. For $d\ge4$, we prove a gap-stability result: every cycle with support less than $(2-10/q)$ times the minimum is a multiple of the boundary of a $d+1$ space.

We also construct support-controlled cones, yielding a direct geometric analysis of the dimensions in which the homology groups of the subspace incidence complex vanish and explicit lower bounds on its coboundary expansion. The degree-$1$ expansion estimate is an ingredient in the stability theorem.
\end{abstract}

\section{Introduction}

A classical theme in combinatorics is the passage from subsets of a finite set to subspaces of a finite vector space. It replaces the Boolean lattice and binomial coefficients by the subspace lattice and Gaussian binomial coefficients. Landmark results include the vector-space Ramsey theorem of Graham, Leeb, and Rothschild \cite{GrahamLeebRothschild1972}, Hsieh's sharp vector-space analogue of the Erd\H{o}s--Ko--Rado theorem \cite{Hsieh1975}, the construction of the first nontrivial $q$-Steiner systems $\mathcal{S}_2(2,3,13)$ by Braun, Etzion, {\"O}sterg{\aa}rd, Vardy, and Wassermann \cite{BraunEtAl2016}, and, more recently, the general existence theorem of Keevash, Sah, and Sawhney for subspace designs in sufficiently large ambient dimension \cite{KeevashSahSawhney2025}.

The basic object in this paper is $\partial_d$, the $0$--$1$ incidence matrix between $d$-spaces and $(d-1)$-spaces.
When regarded over a field $K$ satisfying $(q+1)1_K=0$\footnote{This is equivalent to requiring $char(K) \mid q+1$, e.g. $K=\mathbb{F}_2$ and $q$ odd.},  $\partial_{d-1}\partial_d=0$ for all $d$. The resulting square-zero projective incidence complex was studied by Fisk and by Mnukhin and Siemons \cite{Fisk1997,MnukhinSiemons2000}. Using the rank formula of subspace incidence maps by Frumkin and Yakir \cite{FrumkinYakir1990}, Mnukhin and Siemons determined its homology \cite{MnukhinSiemons2000} \footnote{They considered a more general case where instead of having $\partial^2=0$, $\partial$ is just nilpotent, and studied when the homology vanishes. Consequently, Corollary 3.2 and Theorem 3.7 in the present paper are alternative proof of their result in the $\partial^2=0$ setup.}. These results determine dimensions of kernels and possible homology groups, but not the support of individual kernel vectors.

We write $Z_d(V;K)=\ker\partial_d$ and call its elements $d$-cycles. Their \emph{cycle girth} is
\[
    \gamma_d(V;K)
    :=
    \min\bigl\{
        |\operatorname{supp}(\alpha)|:
        0\neq\alpha\in Z_d(V;K)
    \bigr\}.
\]
The prototypical comparison is the Boolean lattice equipped with the mod-$2$ incidence map from $d$-sets to $(d-1)$-sets. Its cycle girth is $d+1$, attained by the boundary of a simplex on $d+1$ vertices---that is, by the collection of all $d$-sets contained in some fixed $(d+1)$-set. For the $q$-analogues, over $\mathbb R$, Cho determined the minimum support of adjacent-rank null designs and classified the extremizers \cite{Cho1998,Cho1999}. His work gives a $q$-analog for the work of Frankl and Pach \cite{FranklPach1983NullDesign}; Krotov extended the corresponding minimum-volume theorem to subspace trades between non-adjacent ranks \cite{Krotov2017}. The work of Krotov \cite{Krotov2020NullDesigns} studied the structure of cycles in the case where $q 1_K = 0$, which exhibit and entirely different structure and results, since successive applications of $\partial$ are never zero.

\subsection{Our results}

As a main result, we completely characterize $\gamma_d(\F_q^n;K)$ for all $d,q,n$. Surprisingly, the boundary of a single subspace does not always achieve the minimum. For $a\ge1$, let $[a]_q=1+q+\cdots+q^{a-1}$, and set $[0]_q=0$.

\begin{theorem}[Cycle girth]\label{thm:minimum-support}
Let $q$ be a prime power, let $K$ be a field with $(q+1)1_K=0$, and let $1\le d<n$. Then
\[
\gamma_d(\F_q^n;K)
=
\begin{cases}
2, & d=1,\\
[3]_q, & d=2,\ n=3,\\
2(q+1), & d=2,\ n\ge4,\\
[d+1]_q, & d\ge3.
\end{cases}
\]
For $d\ge3$, the boundary $\partial_{d+1}[E]$ of every $(d+1)$-space $E\le\F_q^n$ attains the minimum.
\end{theorem}

In degree $1$, the difference of two distinct $1$-spaces (i.e. the difference of two projective points) is an exceptional cycle. In degree $2$, an appropriately signed sum of isotropic spaces of a $4$-dimensional orthogonal space is a cycle of support $2(q+1)$. For more details, see e.g. \cite{James1984}.

For $d\ge4$, we prove a \emph{gap-stability} theorem, showing that every cycle which is not a boundary of a single $(d+1)$-space must be twice larger than the minimizer (up to low order terms).

\begin{theorem}[Gap stability]\label{thm:gap-stability}
Let $K$ be a field with $(q+1)1_K=0$, let $d\ge4$, and let
$n\ge d+1$. If
\[
0\ne\alpha\in Z_d(\F_q^n;K),
\qquad
|\alpha|<
\left(2-\frac{10}{q}\right)
\gamma_d(\F_q^n;K),
\]
then $\alpha=c\,\partial_{d+1}[E]$ for some $(d+1)$-dimensional
$E\le\F_q^n$ and $c\in K^\times$.
\end{theorem}

We also develop a support-controlled cone construction for the modular incidence complex. Cone constructions, introduced in the seminal works of Linial and Meshulam, Gromov, and Meshulam and Wallach \cite{LinialMeshulam2006,MeshulamWallach2009,Gromov2010} give explicit fillings of cycles with quantitative support bounds, which implies coboundary expansion. They are known to imply coboundary expansion for chain complexes with enough symmetry, see \cite{LubotzkyMeshulamMozes2016}. Due to the action of the general linear group on the grassmannian, our cone bounds give explicit lower bounds on coboundary expansion, which are an ingredient in the proof of gap stability.

The same cone identity gives a new elementary proof that the incidence complex is exact in odd ambient dimension and that, in even ambient dimension, its homology vanishes outside the middle degree; see Corollary~\ref{cor:odd-dimensional-exactness} and Theorem~\ref{thm:homology-away-from-middle}.

\subsubsection*{Further directions}

Theorem~\ref{thm:gap-stability} approaches, but does not determine, the natural two-boundary threshold. Concretely, we ask: Is it the case that for $d>2$, every $0\ne\alpha\in Z_d(\F_q^n;K)$ with
\[
    |\alpha|<2\gamma_d(\F_q^n;K)-1
\]
is of the form $\alpha=\partial\beta$ for some $\beta$ supported on at most two $(d+1)$-spaces? 

Another direction is motivated by the homological consequences of the cone construction, which suggest a second extremal problem. Corollary~\ref{cor:odd-dimensional-exactness} and Theorem~\ref{thm:homology-away-from-middle} show that $H_d(\F_q^n;K)=0$ if and only if $n \neq 2d$. This motivates studying the middle-dimensional systole
\[
    \operatorname{sys}_d(\F_q^{2d};K)
    :=
    \min\bigl\{
        |\operatorname{supp}(\alpha)|:
        \alpha\in Z_d(\F_q^{2d};K)
        \setminus\operatorname{im}\partial_{d+1}
    \bigr\}.
\]
The basic question, analogous to Theorem~\ref{thm:minimum-support}, is to determine this quantity. We conjecture:

\begin{conjecture}[Middle-dimensional systole]
\label{conj:middle-dimensional-systole}
For every $d\ge1$, every prime power $q$, and every field $K$ satisfying $(q+1)1_K=0$,
\[
    \operatorname{sys}_d(\F_q^{2d};K)
    =
    \prod_{i=0}^{d-1}(1+q^i).
\]
\end{conjecture}

We believe that the chain formed by maximally isotropic spaces of an orthogonal form, appearing in the work of James  \cite{James1984} achieves the minimum (one can also show it is not in $\operatorname{im} \partial_{d+1}$).

\paragraph{Organization of the paper.}
The preliminaries below give notation for the projective incidence complex. Section~2 proves Theorem~\ref{thm:minimum-support}, including the exceptional low-dimensional cases. Section~\ref{sec:cones-and-expansion} develops the cone construction, derives its homological and coboundary-expansion consequences, and proves Theorem~\ref{thm:gap-stability}. Section 4 examines the middle degree, where our cone construction stops: it shows that the associated fillings we construct are of optimal size in one dimension higher.

\paragraph{AI disclosure.}

Lemma 2.3 was formulated and proved by GPT 5.6 Sol, and is the key to the proof of Theorem~\ref{thm:minimum-support}. Following this proof, the authors and GPT 5.6 Sol developed the stability version, Theorem~\ref{thm:gap-stability}. The authors have used chatGPT and Codex to assist with writing the paper, and take full responsibility for its mathematical content. Except the above, the rest of the results, and the proof strategies appearing in the paper were found by the authors, with most of them appearing in the Master's Thesis of the second author.

\paragraph{Acknowledgment.}

We thank Jonathan Zin for his generous help with many computer simulations and fruitful discussions, which led to establishing  connections with James's work. The second author thanks Itai Benjamini for various encouragements, and for many insightful conversations. A previous version of this paper (arXiv:2306.14317) has been withdrawn \footnote{It contained an erroneous proof of a conjecture of Mnukhin and Siemons on $ker \partial_d$ over general coefficient rings $R$. Even though by now we have a proof of that conjecture when $q$ is invertible in $R$, we wanted to keep the paper focused on the extremal-combinatorics nature of our research. For the same reason, some previous results on topological overlap/ random $q$-complexes are omitted. }.

\subsection{Preliminaries}

Throughout, $q$ is a prime power and $K$ is a field satisfying
$(q+1)1_K=0$. Unless stated otherwise, $V$ denotes an
$n$-dimensional $\F_q$-space. For any finite-dimensional
$\F_q$-space $X$, let $\Gr_d(X)$ denote the set of its
$d$-dimensional linear subspaces, and let $C_d(X;K)$ be the
$K$-vector space with basis $\{[U]:U\in\Gr_d(X)\}$. For a chain
$\alpha=\sum_Ua_U[U]$, its support and support size are
\[
    \supp(\alpha)=\{U:a_U\ne0\},
    \qquad
    |\alpha|=|\supp(\alpha)|.
\]
If $Y\le X$, we regard $C_d(Y;K)$ as the corresponding coordinate
subspace of $C_d(X;K)$.
For $d\ge1$, define
\[
    \partial_d[U]
    =
    \sum_{W\in\Gr_{d-1}(U)}[W],
\]
and set $\partial_0=0$. 
\begin{observation}
We have $\partial^2=0$.
\end{observation}
\begin{proof}
For $d\ge2$, each $(d-2)$-subspace of a $d$-space is counted exactly $q+1$ times in $\partial_{d-1}\partial_d$, and $(q+1)1_K=0$.
\end{proof}
\noindent From now one we assume $(q+1)1_K=0$, hence we write
\[
    Z_d(X;K)=\ker\partial_d,\qquad
    B_d(X;K)=\operatorname{im}\partial_{d+1},\qquad
    H_d(X;K)=Z_d(X;K)/B_d(X;K).
\]
For $d\ge0$, let
$C^d(X;K)=\operatorname{Hom}_K(C_d(X;K),K)$, and let $\delta_U$
denote the coordinate cochain dual to $[U]$. Set $C^{-1}(X;K)=0$.
The coboundary
$d_k:C^k(X;K)\to C^{k+1}(X;K)$ is the dual of
$\partial_{k+1}$, and we set $d_{-1}=0$. Thus
$(d_k\varphi)([Y])=\varphi(\partial_{k+1}[Y])$. We write
\[
    Z^k(X;K)=\ker d_k,\qquad
    B^k(X;K)=\operatorname{im}d_{k-1},\qquad
    H^k(X;K)=Z^k(X;K)/B^k(X;K).
\]
For a cochain $\varphi=\sum_Ua_U\delta_U$, its support and support
size are defined by the same formulas as for chains.

When $K$ is understood, we suppress it from all chain and cochain
spaces above. When the degree is clear, we write simply $\partial$.

\section{Minimum support of cycles}

Throughout this section, let $V=\F_q^n$. We begin with two lemmas about quotient links of cycles.

\begin{lemma}[Quotient links of cycles are cycles]\label{lem:link}
Let $r\ge2$, let
$\alpha=\sum_{U\in\Gr_r(V)}a_U[U]\in Z_r(V)$, and let
$\ell\le V$ be a $1$-space. Define the quotient link of $\alpha$ at
$\ell$ by
\[
 \alpha_\ell
 :=
 \sum_{\substack{U\in\supp(\alpha)\\ \ell\le U}}
 a_U[U/\ell].
\]
Then $\alpha_\ell\in Z_{r-1}(V/\ell)$.
The map $U\mapsto U/\ell$ is bijective on the subspaces containing
$\ell$. In particular, $\alpha_\ell\ne0$ whenever some member of
$\supp(\alpha)$ contains $\ell$. 
\end{lemma}

\begin{proof}
Writing $\overline P=P/\ell$, the $(r-2)$-spaces of $V/\ell$
correspond to the $(r-1)$-spaces $P\le V$ containing $\ell$. Hence
the coefficient of $[\overline P]$ in $\partial\alpha_\ell$ is
$\sum_{U>P}a_U$, which is also the coefficient of $[P]$ in
$\partial\alpha$ and is therefore zero.

The quotient map is a bijection between the subspaces of $V$
containing $\ell$ and the subspaces of $V/\ell$. 
\end{proof}

\begin{lemma}[Counting quotient links]
\label{lem:counting-quotient-links}
Let $d\ge2$, let $0\ne\alpha\in Z_d(V)$, put
$S=\supp(\alpha)$, and fix $U_0\in S$. Then
\[
 \sum_{\ell\in\Gr_1(U_0)}|\alpha_\ell|
 =
 [d]_q+
 \sum_{U\in S\setminus\{U_0\}}[\dim(U_0\cap U)]_q.
\]
\end{lemma}

\begin{proof}
Both sides count the pairs $(\ell,U)$ with $U\in S$ and
$\ell\in\Gr_1(U_0\cap U)$.
\end{proof}

\subsection{A relative estimate for two-cycles}
\begin{lemma}[Intersection-sensitive lower bound]
\label{lem:relative-two-cycle}
Assume $n\ge3$, let
$\beta=\sum_{U\in\Gr_2(V)}b_U[U]\in Z_2(V)$, and fix
$U_0\in\supp(\beta)$. Put
\[
 \mathcal X=
 \{U\in\supp(\beta)\setminus\{U_0\}:\dim(U\cap U_0)=1\},
 \qquad
 \mathcal Y=
 \{U\in\supp(\beta):U\cap U_0=0\}.
\]
Then
\[
 |\mathcal X|+(q+1)|\mathcal Y|\ge q(q+1).
\]
\end{lemma}

\begin{proof}
We use projective terminology: $1$-spaces are points, $2$-spaces are
lines, and $3$-spaces are planes. Thus $U_0$ is a projective line, the
members of $\mathcal X$ meet $U_0$, and the members of $\mathcal Y$ are
skew to $U_0$.

Let $\mathcal D=\Gr_1(U_0)$, so $|\mathcal D|=q+1$. If
$H\supset U_0$ is a projective plane and $d\in\mathcal D$, define
$\mathcal L_H(d)=\{U\in\Gr_2(H):U\cap U_0=d\}$. Viewing $U_0$ as the
line at infinity of $H$, the set $\mathcal L_H(d)$ is the parallel
class of affine lines of direction $d$; in particular,
$|\mathcal L_H(d)|=q$.

For such $H$ and $d$, put
\[
 T_H^{\ne d}
 =
 \sum_{\substack{X\in\mathcal X\\X\le H,\ X\cap U_0\ne d}} b_X.
\]
Thus $T_H^{\ne d}$ records the total contribution of $\mathcal{X} \cap Gr_2(H)$ to the cycle equations in all directions in $\mathcal D \setminus \{d\}$.

\begin{claim}\label{cl:choose-plane}
For every $d\in\mathcal D$, there is a projective plane $H(d)\supset U_0$ such
that $T_{H(d)}^{\ne d}\ne0$.
\end{claim}

\begin{proof}
For each $e\in\mathcal D$, the supported lines containing $e$ are
$U_0$ and the members $X\in\mathcal X$ with $X\cap U_0=e$. Hence the
cycle equation at $e$ is
\[
 b_{U_0}+\sum_{\substack{X\in\mathcal X\\X\cap U_0=e}}b_X=0.
\]
Fix $d$ and sum these equations over the $q$ points $e\ne d$. Every
$X\in\mathcal X$ with $X\cap U_0\ne d$ belongs to the unique plane
$U_0+X$. Hence
\begin{align*}
 \sum_{\substack{H\in\Gr_3(V)\\U_0<H}}T_H^{\ne d}
 &=
 \sum_{e\ne d}\ \sum_{\substack{X\in\mathcal X\\X\cap U_0=e}}b_X\\
 &=-q b_{U_0}=b_{U_0}\ne0.
\end{align*}
The last equality uses $(q+1)1_K=0$.
At least one summand is therefore nonzero.
\end{proof}

We now show that, for each $d\in\mathcal D$, every line in
$\mathcal L_{H(d)}(d)$ either belongs to $\mathcal X$ or meets a
member of $\mathcal Y$.
\begin{claim}\label{cl:parallel-line}
Fix $d\in\mathcal D$. Every line
$U\in\mathcal L_{H(d)}(d)$ either belongs to $\mathcal X$ or is met
at an affine point by some $Y\in\mathcal Y$.
\end{claim}

\begin{proof}
Sum the cycle equations at the $q$ affine points of $U$. Their sum is
zero. The line $U_0$ contains none of these points. Moreover, if
$F\in\mathcal X$ contains one of them, then $F\le H(d)$, since both
that point and $F\cap U_0$ lie in $H(d)$. We group the remaining
contributions as follows.

\begin{description}
\item[Direction $d$.]
The line $U$ contributes $q b_U$, once for each of its affine points.
Every other line in $H(d)$ of direction $d$ is parallel to $U$ and
contributes nothing.

\item[Other directions in $H(d)$.]
Every $F\in\mathcal X$ with $F\le H(d)$ and $F\cap U_0\ne d$ meets $U$
in exactly one affine point. These lines contribute
$T_{H(d)}^{\ne d}$ in total.

\item[Lines in $\mathcal Y$.]
A line $Y\in\mathcal Y$ cannot lie in $H(d)$, since every line in the
plane $H(d)$ meets $U_0$. Hence $Y\cap H(d)$ is at most one affine
point. Their total contribution is the double sum below.
\end{description}

Consequently,
\[
 q b_U+T_{H(d)}^{\ne d}
 +
 \sum_{\substack{r\in\Gr_1(U)\\r\nleq U_0}}
 \ \sum_{\substack{Y\in\mathcal Y\\r\le Y}} b_Y
 =0.
\]
If $U\notin\mathcal X$, then $b_U=0$. Since
$T_{H(d)}^{\ne d}\ne0$, the displayed identity implies that the
double sum over $\mathcal Y$ is nonzero. Hence some $Y\in\mathcal Y$
meets $U$ at an affine point.
\end{proof}

We now count the selected lines. For each $d\in\mathcal D$, consider
the $q$ lines in the parallel class $\mathcal L_{H(d)}(d)$. Altogether
this gives $q(q+1)$ selected lines. By Claim~\ref{cl:parallel-line}, each
of them either belongs to $\mathcal X$ or is met by a member of
$\mathcal Y$.

A fixed $X\in\mathcal X$ can occur among the selected lines at most
once: its direction is the unique point $X\cap U_0$, and the plane
containing both $U_0$ and $X$ is $U_0+X$.

A fixed $Y\in\mathcal Y$ meets at most one selected line for each
$d\in\mathcal D$. Indeed, $Y$ is not contained in $H(d)$, since every
projective line in the plane $H(d)$ meets $U_0$. Thus $Y\cap H(d)$ is
at most one affine point, and through that point there is exactly one
line of direction $d$. Since there are $q+1$ directions in total,
$Y$ meets at most $q+1$ selected lines. Counting the selected lines
gives $q(q+1)\le |\mathcal X|+(q+1)|\mathcal Y|$, as required.
\end{proof}

\subsection{Degrees three and higher}

We first prove Theorem~\ref{thm:minimum-support} for $d=3$.
\begin{proposition}[The degree-$3$ case]\label{prop:d3}
Let $n\ge4$ and let $0\ne\alpha\in Z_3(V)$. Then
\[
 |\alpha|\ge[4]_q=q^3+q^2+q+1.
\]
\end{proposition}

\begin{proof}
Fix $U_0\in\supp(\alpha)$. For $i=0,1,2$, let
\[
 \mathcal S_i=
 \{U\in\supp(\alpha)\setminus\{U_0\}:\dim(U\cap U_0)=i\},
 \qquad m_i=|\mathcal S_i|.
\]
For every $1$-space $\ell\le U_0$, Lemma~\ref{lem:link} gives a
nonzero $2$-cycle $\alpha_\ell$ in $V/\ell$ such that
$U_0/\ell\in\supp(\alpha_\ell)$. Define
\[
 x_\ell=
 |\{U\in\mathcal S_2:\ell\le U\cap U_0\}|,
 \qquad
 y_\ell=
 |\{U\in\mathcal S_1:U\cap U_0=\ell\}|.
\]
For every $U\in\supp(\alpha)$ containing $\ell$, we have
$(U/\ell)\cap(U_0/\ell)=(U\cap U_0)/\ell$. Hence
Lemma~\ref{lem:relative-two-cycle} gives
$x_\ell+(q+1)y_\ell\ge q(q+1)$ for every $\ell\le U_0$.

Sum this inequality over the $[3]_q$ $1$-spaces of $U_0$. Each member
of $\mathcal S_2$ is counted once for every $1$-space in its
$2$-dimensional intersection with $U_0$, hence $q+1$ times; each
member of $\mathcal S_1$ is counted once. Thus
$(q+1)(m_1+m_2)\ge[3]_q q(q+1)$, and therefore
$m_1+m_2\ge q[3]_q$. It follows that
\[
 |\alpha|
 =1+m_0+m_1+m_2
 \ge1+q[3]_q=[4]_q.
\]
\end{proof}

\begin{proof}[Proof of Theorem~\ref{thm:minimum-support} for $d\ge3$]
We prove the lower bound by induction on $d$.
Proposition~\ref{prop:d3} is the base case. Fix $d\ge3$, assume the
result holds in degree $d$ in every ambient dimension, and let
$0\ne\alpha\in Z_{d+1}(V)$.
Fix $U_0\in\supp(\alpha)$ and put $m=|\alpha|$. For every
$1$-space $\ell\le U_0$, the quotient link $\alpha_\ell$ is a
nonzero $d$-cycle in $V/\ell$, so
$|\alpha_\ell|\ge[d+1]_q$.
Lemma~\ref{lem:counting-quotient-links} gives
\begin{align*}
 [d+1]_q^2
 &\le
 \sum_{\ell\in\Gr_1(U_0)}|\alpha_\ell|\\
 &=
 [d+1]_q+
 \sum_{U\in\supp(\alpha)\setminus\{U_0\}}
 [\dim(U\cap U_0)]_q.
\end{align*}
Every $U\ne U_0$ satisfies $\dim(U\cap U_0)\le d$. Hence
$[d+1]_q^2\le[d+1]_q+(m-1)[d]_q$.
Using $[d+1]_q-1=q[d]_q$, we obtain
$m\ge1+q[d+1]_q=[d+2]_q$.
This proves the lower bound by induction.

For sharpness in any degree $3\le k<n$, let $E\le V$ have dimension
$k+1$.
Every $(k-1)$-space in $E$ lies in exactly $q+1$ of its $k$-spaces,
so $\partial_k\partial_{k+1}[E]=0$. Moreover,
\[
 \partial_{k+1}[E]
 =\sum_{D\in\Gr_k(E)}[D],
 \qquad
 |\Gr_k(E)|=[k+1]_q.
\]
Thus $\partial_{k+1}[E]$ attains the bound.
\end{proof}

\subsection{The low-dimensional cases}

We now classify the minimum support size of cycles in degrees less
than $3$.
\begin{proposition}[Low-dimensional cycles]
\label{prop:low-dimensional-cycles}
The following hold:
\begin{enumerate}
    \item $d=1$: $Z_1(V)=0$ for $n\le1$, while
    $\min_{0\ne\alpha\in Z_1(V)}|\alpha|=2$ for $n\ge2$.
    \item $d=2$: $Z_2(V)=0$ for $n\le2$, while
    \[
     \min_{0\ne\alpha\in Z_2(V)}|\alpha|
     =
     \begin{cases}
        [3]_q,&n=3,\\
        2(q+1),&n\ge4.
     \end{cases}.
    \]
\end{enumerate}
\end{proposition}

\begin{proof}
A $1$-chain $\sum_L a_L[L]$ is a cycle precisely when $\sum_L a_L=0$. Thus a
nonzero $1$-cycle contains at least two $1$-spaces, and
$[L_1]-[L_2]$ attains the bound for $n\ge2$.

For $n<2$ there are no $2$-spaces, while for $n=2$ the unique
$2$-space has nonzero boundary. Suppose that $n=3$, and fix
$0\ne\alpha=\sum_{P\in\Gr_2(V)}a_P[P]\in Z_2(V)$. For each
$P_0\in\Gr_2(V)$, summing the cycle equations over the $q+1$
$1$-spaces of $P_0$ gives
$0=(q+1)a_{P_0}+\sum_{P\ne P_0}a_P=\sum_{P\ne P_0}a_P$.
Thus, if $s=\sum_Pa_P$, then $a_{P_0}=s$ for every $P_0$.
Consequently, $\alpha=s\,\partial_3[V]$, and $|\alpha|=[3]_q$.

It remains to consider $n\ge4$. Fix
$0\ne\alpha=\sum_{P\in\Gr_2(V)}a_P[P]\in Z_2(V)$. For a $3$-space
$H\le V$, put
$\alpha_H=\sum_{P\in\Gr_2(H)}a_P[P]$ and $s_H=|\alpha_H|$.
Some $H$ satisfies $s_H\ge2$: indeed, if
$P\in\supp(\alpha)$ and $L<P$ is a $1$-space, then the cycle equation
at $L$ gives another supported $2$-space $P'\ne P$ containing $L$,
and $H=P+P'$ contains both. Define
$f=\min\{s_H:H\in\Gr_3(V),\ s_H\ge2\}$.

Assume first that $f\ge q+2$, and choose $H$ with $s_H=f$. If
$\partial\alpha_H=0$, then the case $n=3$ proved above gives
$s_H=[3]_q$, and hence $|\alpha|\ge[3]_q\ge2(q+1)$. Otherwise choose
a $1$-space
$L\in\supp(\partial\alpha_H)$. Since $\partial\alpha=0$, there is
$P\in\supp(\alpha)\setminus\Gr_2(H)$ containing $L$. Choose also
$Q\in\supp(\alpha_H)$ containing $L$, and put $H'=P+Q$. Then
$H'\in\Gr_3(V)$, the spaces $P$ and $Q$ belong to
$\supp(\alpha_{H'})$, and $H\cap H'=Q$. Thus $s_{H'}\ge2$, so the
definition of $f$ gives $s_{H'}\ge f$. Therefore
$|\alpha|\ge s_H+s_{H'}-1\ge2f-1>2(q+1)$.

Suppose now that $2\le f\le q+1$, and choose $H$ with $s_H=f$. For
a $1$-space $L\le H$, let $t_L$ be the number of members of
$\supp(\alpha_H)$ containing $L$. Then
\[
 \sum_{L\in\Gr_1(H)}t_L=f(q+1),
 \qquad
 \sum_{L\in\Gr_1(H)}\binom{t_L}{2}=\binom f2,
\]
because two distinct $2$-spaces in $H$ meet in a unique $1$-space. Observe that a
$1$-space with $t_L=1$ belongs to $\supp(\partial\alpha_H)$. Consequently, 
\begin{align*}
 |\partial\alpha_H|
 &\ge
 \sum_{L\in\Gr_1(H)}\left(t_L-2\binom{t_L}{2}\right)
 =f(q+2-f).
\end{align*}
Indeed, the summand on the right is $1$ when $t_L=1$, vanishes when
$t_L\in\{0,2\}$, and is nonpositive when $t_L\ge3$.

Since $\partial\alpha=0$, every $1$-space in
$\supp(\partial\alpha_H)$ is contained in a supported $2$-space
outside $H$. Such a $2$-space contains at most one $1$-space of $H$,
and therefore
\begin{align*}
 |\alpha|
 &\ge f+f(q+2-f)\\
 &=2(q+1)+(f-2)(q+1-f)\\
 &\ge2(q+1).
\end{align*}

The chain $[L_1]-[L_2]$ attains the degree-$1$ bound. For $d=2$ and
$n=3$, the boundary $\partial_3[V]$ attains the bound $[3]_q$.
For the $d=2$ and $n=4$ it is easy to verify that the following chain is a cycle: \[\sum_{\lambda \in F_q \cup \infty} \langle e_1+\lambda e_2, e_3+\lambda e_4 \rangle  - \sum_{\lambda \in F_q \cup \infty} \langle e_1+\lambda e_3, e_2+\lambda e_4 \rangle.\]
\end{proof}

Together with the case $d\ge3$ proved above, this completes the proof
of Theorem~\ref{thm:minimum-support}.

\subsection{A quotient-link defect estimate}

We record a defect identity and a consequence of it that will be used in
the proof of Theorem~\ref{thm:gap-stability}.

\begin{lemma}[Quotient-link defect identity]
\label{lem:quotient-link-defect}
Let $d\ge4$, let $0\ne\alpha\in Z_d(V)$, put $S=\supp(\alpha)$, and
set $\delta=|S|-[d+1]_q\ge0$, where nonnegativity follows from
Theorem~\ref{thm:minimum-support}. For every $U_0\in S$,
\begin{equation}\label{eq:quotient-link-defect}
 \delta[d-1]_q
 =
 \sum_{U\in S\setminus\{U_0\}}
 \bigl([d-1]_q-[\dim(U_0\cap U)]_q\bigr)
 +
 \sum_{\ell\in\Gr_1(U_0)}\bigl(|\alpha_\ell|-[d]_q\bigr).
\end{equation}
\end{lemma}

\begin{proof}
Substituting the identity from
Lemma~\ref{lem:counting-quotient-links} into the right-hand side of
\eqref{eq:quotient-link-defect} gives
\begin{align*}
&\sum_{U\in S\setminus\{U_0\}}
 \bigl([d-1]_q-[\dim(U_0\cap U)]_q\bigr)
 +\sum_{\ell\in\Gr_1(U_0)}\bigl(|\alpha_\ell|-[d]_q\bigr)\\
&\qquad=(|S|-1)[d-1]_q+[d]_q-[d]_q^2\\
&\qquad=\delta[d-1]_q.
\end{align*}
The last equality uses $[d]_q-1=q[d-1]_q$.
\end{proof}

\begin{corollary}\label{cor:quotient-link-defect}
Under the assumptions of Lemma~\ref{lem:quotient-link-defect}, for
every $U_0\in S$,
\[
 \bigl|\{U\in S\setminus\{U_0\}:
 \dim(U_0\cap U)\le d-2\}\bigr|
 \le \frac{\delta[d-1]_q}{q^{d-2}}
 \le\left(1+\frac2q\right)\delta.
\]
\end{corollary}

\begin{proof}
If $U\ne U_0$, then $\dim(U_0\cap U)\le d-1$. Moreover, every quotient
link $\alpha_\ell$ with $\ell\in\Gr_1(U_0)$ is nonzero by
Lemma~\ref{lem:link}, so Theorem~\ref{thm:minimum-support}, applied in
degree $d-1$, gives $|\alpha_\ell|\ge[d]_q$. Each $U$ with
$\dim(U_0\cap U)\le d-2$ contributes at least
$[d-1]_q-[d-2]_q=q^{d-2}$ to the first sum in
\eqref{eq:quotient-link-defect}. 
Finally,
\[
\frac{[d-1]_q}{q^{d-2}}
=1+\frac1q+\cdots+\frac1{q^{d-2}}
\le1+\frac1{q-1}
\le1+\frac2q.
\]
\end{proof}

\section{Cones, expansion, and stability}
\label{sec:cones-and-expansion}

For a finite-dimensional $\F_q$-space $X$ and $j\ge0$, we consider the coboundary operator on $X$,
\[
 d_j^X:C_j(X)\longrightarrow C_{j+1}(X),
 \qquad
 d_j^X[U]
 =
 \sum_{\substack{Y\in\Gr_{j+1}(X)\\U<Y}}[Y].
\]
The superscript specifies the ambient space in which the cofaces are
taken. Set $d_{-1}^X=0$.
The following identity is the main algebraic input to the
cone construction.

\begin{lemma}[Incidence commutator]
\label{lem:incidence-commutator}
Let $M$ be an $m$-dimensional $\F_q$-space. For every
$0\le j\le m$,
\[
 \partial_{j+1}d_j^M-d_{j-1}^M\partial_j
 =
 \bigl([m-j]_q-[j]_q\bigr)I
\]
on $C_j(M)$. If $\dim M=2j+1$ and $z\in Z_j(M)$, then
\[
 \partial\bigl(q^{-j}d_j^Mz\bigr)=z.
\]
\end{lemma}

\begin{proof}
Fix $U,W\in\Gr_j(M)$. If $W=U$, the coefficient of $[W]$ in
$\partial d_j^M[U]$ is $[m-j]_q$, while its coefficient in
$d_{j-1}^M\partial[U]$ is $[j]_q$. If $W\ne U$, both coefficients
are $1$ when $\dim(U\cap W)=j-1$ and are $0$ otherwise. Thus the
off-diagonal coefficients cancel.

If $m=2j+1$ and $z\in Z_j(M)$, the commutator identity and
$\partial z=0$ give $\partial d_j^Mz=q^jz$. Since $q=-1$ in $K$,
multiplication by $q^{-j}$ proves the second assertion.
\end{proof}

The last lemma proved that when $M$ is odd dimensional, every cycle $z$ is filled by $q^{-j}d_j^Mz$, consequently we get:
\begin{corollary}[Odd-dimensional exactness]
\label{cor:odd-dimensional-exactness}
If $\dim V$ is odd, then $H_t(V)=0$ for every $t$.
\end{corollary}

\subsection{The cone construction}

\begin{definition}[Cone]
A cone on $V$ through degree $k$ is a family of linear maps
$c_j:C_j(V)\to C_{j+1}(V)$, $0\le j\le k$, such that, with the
convention $c_{-1}=0$,
\[
 \partial c_j+c_{j-1}\partial=I_{C_j(V)}
 \qquad(0\le j\le k).
\]
\end{definition}

Let $\mathbf b=(b_1,\ldots,b_n)$ be an ordered basis of $V$. For
$W\in\Gr_j(V)$ with $2j+1\le n$, let $r_{\mathbf b}(W)$ be the first
index $r$ for which
$\dim(W+\langle b_1,\ldots,b_r\rangle)=2j+1$, and set
\[
 M_{\mathbf b}(W)
 =
 W+\langle b_1,\ldots,b_{r_{\mathbf b}(W)}\rangle.
\]
Set $c_{\mathbf b,-1}=0$. For every $j\ge0$ with $2j+1\le n$ and
every $W\in\Gr_j(V)$, define
\[
 c_{\mathbf b,j}[W]
 =
 q^{-j}d_j^{M_{\mathbf b}(W)}
 \bigl([W]-c_{\mathbf b,j-1}(\partial[W])\bigr),
\]
and extend linearly to $C_j(V)$.
The following proposition verifies that these recursively defined maps are indeed well defined.

\begin{proposition}[Cone identity]\label{prop:cone-identity}
For every $j\ge0$ with $2j+1\le n$ and every $x\in C_j(V)$,
\[
 \partial c_{\mathbf b,j}(x)
 =
 x-c_{\mathbf b,j-1}(\partial x).
\]
Moreover, $c_{\mathbf b,j}[W]$ is supported in
$M_{\mathbf b}(W)$ for every $j\ge0$ with $2j+1\le n$ and every
$W\in\Gr_j(V)$.
\end{proposition}

\begin{proof}
We argue by induction on $j$. For $j=0$, we have
$M_{\mathbf b}(\langle0\rangle)=\langle b_1\rangle$ and
$c_{\mathbf b,0}[\langle0\rangle]=[\langle b_1\rangle]$, whose
boundary is $[\langle0\rangle]$. Thus the degree-zero cone identity
and the support assertion hold.

Suppose $j\ge1$. Fix $W\in\Gr_j(V)$, let $U<W$ be a hyperplane, and
write $r=r_{\mathbf b}(W)$. Then
$\dim(U+\langle b_1,\ldots,b_r\rangle)\ge2j$. The dimension of
$U+\langle b_1,\ldots,b_s\rangle$ increases by at most one when $s$
increases by one. Therefore the first prefix for which it has
dimension $2j-1$ has length at most $r$, and hence
$M_{\mathbf b}(U)\le M_{\mathbf b}(W)$. The induction hypothesis
therefore shows that
$c_{\mathbf b,j-1}(\partial[W])$ is supported in
$M_{\mathbf b}(W)$, so the recursive definition is well defined and
has the asserted support.

Put $M=M_{\mathbf b}(W)$ and
$z_W=[W]-c_{\mathbf b,j-1}(\partial[W])$. By the preceding support
argument, $z_W\in C_j(M)$, while the induction hypothesis and
$\partial^2=0$ give $\partial z_W=0$. Since $\dim M=2j+1$, the
incidence commutator gives
\[
\begin{aligned}
 \partial c_{\mathbf b,j}[W]
 &=
 q^{-j}\partial d_j^M z_W \\
 &=
 q^{-j}\bigl(d_{j-1}^M\partial z_W+q^jz_W\bigr) \\
 &=
 z_W
 =
 [W]-c_{\mathbf b,j-1}(\partial[W]).
\end{aligned}
\]
Thus the cone identity holds on $[W]$, and linearity completes the
induction.
\end{proof}

In particular, $(c_{\mathbf b,j})_{j=0}^k$ is a cone through degree
$k$ whenever $2k+1\le n$.

\begin{lemma}[Equivariance]\label{lem:cone-equivariance}
Let $g\in\operatorname{GL}(V)$ and write
$g\mathbf b=(gb_1,\ldots,gb_n)$. For every $j\ge0$ with
$2j+1\le n$ and every $x\in C_j(V)$,
\[
 c_{g\mathbf b,j}(gx)=g\,c_{\mathbf b,j}(x).
\]
\end{lemma}

\begin{proof}
For every $W\le V$, the defining prefix for $gW$ with respect to
$g\mathbf b$ is the image of the defining prefix for $W$. Thus
$M_{g\mathbf b}(gW)=gM_{\mathbf b}(W)$. The coboundary and boundary
operators commute with $g$, and the claim follows by induction from
the recursive definition.
\end{proof}

\subsection{Homological consequences}

Fix a nondegenerate symmetric $\F_q$-bilinear form on $V$. For
$x=\sum_Ua_U[U]\in C_j(V)$, put 
\[
 x^\perp
 =
 \sum_Ua_U\,\delta_{U^\perp}
 \in C^{n-j}(V).
\]

\begin{lemma}[Orthogonal-complement duality]
\label{lem:orthogonal-complement-duality}
For every $0\le j\le n$, the map $x\mapsto x^\perp$ is a
support-preserving isomorphism from $C_j(V)$ to $C^{n-j}(V)$ and
satisfies
\[
 (\partial_jx)^\perp=d_{n-j}(x^\perp).
\]
Consequently,
$H_j(V)\cong H^{n-j}(V)$.
\end{lemma}

\begin{proof}
Orthogonal complementation is a bijection from $\Gr_j(V)$ to
$\Gr_{n-j}(V)$. If $T<U$ has codimension one, then
$U^\perp<T^\perp$ has codimension one, and every coface of
$U^\perp$ arises uniquely in this way. This proves the displayed
identity on basis elements. It also shows that cycles correspond to
cocycles and boundaries to coboundaries, yielding the claimed
isomorphism.
\end{proof}

\begin{theorem}[Homology away from the middle dimension]
\label{thm:homology-away-from-middle}
If $\dim V=2m$, then
\[
 H_t(V)=0
 \qquad\text{for every }t\ne m.
\]
\end{theorem}

\begin{proof}
Let $t<m$ and let $z\in Z_t(V)$. Since $2t+1\le2m-1$, the cone map
is defined, and Proposition~\ref{prop:cone-identity} gives
$z=\partial c_{\mathbf b,t}(z)$. Thus $H_t(V)=0$.

If $t>m$, Lemma~\ref{lem:orthogonal-complement-duality} gives
$H_t(V)\cong H^{2m-t}(V)$. By the universal coefficient theorem (as we work over a field $K$), $H^{2m-t}\cong H_{2m-t}$ which is trivial as $2m-t<m$ for $t>m$.
\end{proof}

\subsection{Coboundary expansion}
\label{subsec:coboundary-expansion}

For $\alpha\in C^j(V)$, define its distance from the coboundaries by
\[
 \operatorname{dist}\bigl(\alpha,B^j(V)\bigr)
 =
 \min_{\gamma\in B^j(V)}|\alpha-\gamma|
 =
 \min_{\beta\in C^{j-1}(V)}|\alpha-d_{j-1}\beta|.
\]
For $0\le j<n$, the $j$-dimensional coboundary expansion constant is
\[
 h_{j,n}
 =
 \min_{\alpha\in C^j(V)\setminus B^j(V)}
 \frac{|d_j\alpha|}
 {\operatorname{dist}\bigl(\alpha,B^j(V)\bigr)}.
\]

For an ordered basis $\mathbf b$ and $j\ge0$ with $2j+1\le n$,
define
\[
 \iota_{\mathbf b,j}:C^{j+1}(V)\longrightarrow C^j(V),
 \qquad
 (\iota_{\mathbf b,j}\varphi)(x)
 =
 \varphi(c_{\mathbf b,j}(x)).
\]

\begin{lemma}[Cochain contraction]\label{lem:cochain-contraction}
If $n\ge1$, then
\[
 \iota_{\mathbf b,0}d_0=I_{C^0(V)}.
\]
For every $j\ge1$ with $2j+1\le n$,
\[
 \iota_{\mathbf b,j}d_j
 +
 d_{j-1}\iota_{\mathbf b,j-1}
 =
 I_{C^j(V)}.
\]
\end{lemma}

\begin{proof}
For $j=0$, the assertion is dual to
$\partial c_{\mathbf b,0}=I_{C_0(V)}$. Let $j\ge1$. For
$\alpha\in C^j(V)$ and $x\in C_j(V)$, the left-hand side,
evaluated at $x$, equals
\[
 \alpha\bigl(\partial c_{\mathbf b,j}(x)\bigr)
 +
 \alpha\bigl(c_{\mathbf b,j-1}(\partial x)\bigr)
 =
 \alpha(x)
\]
by Proposition~\ref{prop:cone-identity}.
\end{proof}

Put $[0]_q!=1$ and
$[j]_q!=\prod_{r=1}^j[r]_q$, as ordinary integers. Define
\[
 \kappa_0=1,
 \qquad
 \kappa_j=[j+1]_q\bigl(1+[j]_q\kappa_{j-1}\bigr)
 \quad(j\ge1).
\]

\begin{proposition}[Cone size]\label{prop:cone-size}
For every ordered basis $\mathbf b$, every $j$ with $2j+1\le n$,
and every $W\in\Gr_j(V)$,
\[
 |c_{\mathbf b,j}[W]|\le\kappa_j.
\]
Moreover,
\[
 \kappa_j
 =
 [j+1]_q[j]_q!^2
 \left(
 1+\sum_{r=1}^j\frac{1}{[r]_q!^2}
 \right).
\]
In particular, $\kappa_j$ bounds the cone sizes in every degree at most
$j$.
\end{proposition}

\begin{proof}
The assertion is immediate for $j=0$. Suppose it holds in degree
$j-1$. The support of $\partial[W]$ has size $[j]_q$, and every
$j$-space in the $(2j+1)$-space $M_{\mathbf b}(W)$ has
$[j+1]_q$ cofaces there. Hence
\[
 |c_{\mathbf b,j}[W]|
 \le
 [j+1]_q\bigl(1+[j]_q\kappa_{j-1}\bigr)
 =\kappa_j.
\]
The closed form follows from the recurrence by induction. The same
recurrence shows that the sequence $(\kappa_j)_{j\ge0}$ is increasing.
\end{proof}

\begin{theorem}[Coboundary expansion]
\label{thm:grassmann-coboundary-expansion}
Let $1\le k<n/2$. Suppose that
$|c_{\mathbf b,j}[W]|\le\Lambda$ for every ordered basis $\mathbf b$,
every $0\le j\le k$, and every $W\in\Gr_j(V)$. Then
\[
 h_{k,n}
 \ge
 \frac{[n-k]_q}{[k+1]_q}\frac1\Lambda,
 \qquad
 h_{n-k,n}\ge\frac1\Lambda.
\]
In particular, one may take $\Lambda=\kappa_k$ from
Proposition~\ref{prop:cone-size}.
\end{theorem}

\begin{proof}
Let $\mathcal B(V)$ be the set of ordered bases of $V$. Since $d_k$
vanishes on $B^k(V)$, it is constant on every coset of $B^k(V)$.
Fix a nonzero coset in $C^k(V)/B^k(V)$, and choose a representative
$\alpha$ of minimum support. Thus
\[
 |\alpha|
 =
 \operatorname{dist}\bigl(\alpha,B^k(V)\bigr).
\]
For every $\mathbf b\in\mathcal B(V)$,
Lemma~\ref{lem:cochain-contraction} gives
\[
 \alpha-d_{k-1}\iota_{\mathbf b,k-1}\alpha
 =
 \iota_{\mathbf b,k}d_k\alpha.
\]
The left-hand side represents the same coset as $\alpha$. Hence the
minimality of $\alpha$ implies
\[
 |\alpha|\le|\iota_{\mathbf b,k}d_k\alpha|.
\]

We next relate the support on the right to the cones. For every
$W\in\Gr_k(V)$,
\[
 (\iota_{\mathbf b,k}d_k\alpha)([W])
 =
 (d_k\alpha)(c_{\mathbf b,k}[W]).
\]
If this value is nonzero, then some member of
$\supp(c_{\mathbf b,k}[W])$ has a nonzero coefficient in
$d_k\alpha$. Possible cancellations can only make the support of the
contracted cochain smaller. Therefore
\[
 |\iota_{\mathbf b,k}d_k\alpha|
 \le
 \sum_{W\in\Gr_k(V)}
 \mathbf 1_{\{
 \supp(c_{\mathbf b,k}[W])
 \cap\supp(d_k\alpha)\ne\varnothing\}}.
\]
Summing over all $\mathbf b \in \mathcal B(V) $
\[
 |\mathcal B(V)|\,|\alpha|
 \le
 \sum_{\mathbf b\in\mathcal B(V)}
 \sum_{W\in\Gr_k(V)}
 \mathbf 1_{\{
 \supp(c_{\mathbf b,k}[W])
 \cap\supp(d_k\alpha)\ne\varnothing\}}.
\]

For $U\in\Gr_{k+1}(V)$, let
\[
 N(U)
 =
 \bigl|\{(\mathbf b,W)\in\mathcal B(V)\times\Gr_k(V):
 U\in\supp(c_{\mathbf b,k}[W])\}\bigr|.
\]
Lemma~\ref{lem:cone-equivariance} and the transitive action of
$\operatorname{GL}(V)$ on $\Gr_{k+1}(V)$ show that $N(U)$ is
independent of $U$; denote its common value by $N$. Double counting
the intersections with $\supp(d_k\alpha)$ gives
\begin{align*}
 &\sum_{\mathbf b\in\mathcal B(V)}
 \sum_{W\in\Gr_k(V)}
 \mathbf 1_{\{
 \supp(c_{\mathbf b,k}[W])
 \cap\supp(d_k\alpha)\ne\varnothing\}}\\
 &\qquad\le
 \sum_{U\in\supp(d_k\alpha)}N(U)
 =
 N|d_k\alpha|.
\end{align*}
On the other hand, counting all incidences between cones and
$(k+1)$-spaces yields
\begin{align*}
 |\Gr_{k+1}(V)|N
 &=
 \sum_{U\in\Gr_{k+1}(V)}N(U)\\
 &=
 \sum_{\mathbf b\in\mathcal B(V)}
 \sum_{W\in\Gr_k(V)}
 |c_{\mathbf b,k}[W]|\\
 &\le
 |\mathcal B(V)|\,|\Gr_k(V)|\,\Lambda.
\end{align*}
Combining these estimates and cancelling $|\mathcal B(V)|$ gives
\[
 |\alpha|
 \le
 \frac{|\Gr_k(V)|}{|\Gr_{k+1}(V)|}
 \Lambda\,|d_k\alpha|.
\]
Since
$|\Gr_{k+1}(V)|/|\Gr_k(V)|=[n-k]_q/[k+1]_q$, this proves the first
bound.

For the second bound, fix $\alpha\in C^{n-k}(V)$ and write
$\alpha=\beta^\perp$ for the unique $\beta\in C_k(V)$. Put
$z=\partial\beta\in Z_{k-1}(V)$ and fix an ordered basis $\mathbf b$.
Since $z$ is a cycle, the cone identity gives
$\partial c_{\mathbf b,k-1}(z)=z$, so
$c_{\mathbf b,k-1}(z)$ is a filling of $z$. Applying the cone identity
to $\beta$ also gives
\[
 \beta-c_{\mathbf b,k-1}(z)
 =
 \partial c_{\mathbf b,k}(\beta).
\]
Applying orthogonal complementation and
Lemma~\ref{lem:orthogonal-complement-duality}, we obtain
\[
 \alpha-\bigl(c_{\mathbf b,k-1}(z)\bigr)^\perp
 =
 d_{n-k-1}\bigl(\bigl(c_{\mathbf b,k}(\beta)\bigr)^\perp\bigr)
 \in B^{n-k}(V).
\]
Moreover, $z^\perp=d_{n-k}\alpha$. Since orthogonal complementation
preserves support and cone size, and the cone size is at most $\Lambda$ in degree
$k-1$,
\begin{align*}
 \operatorname{dist}\bigl(\alpha,B^{n-k}(V)\bigr)
 &\le |c_{\mathbf b,k-1}(z)|\\
 &\le
 \sum_{W\in\supp(z)}|c_{\mathbf b,k-1}[W]|\\
 &\le \Lambda|z|
 =\Lambda|d_{n-k}\alpha|.
\end{align*}
\end{proof}

\subsubsection*{Expansion in dimension three}

We proceed with some explicit computations of the expansion when the ambient linear space is of dimension $3$. These computations are standalone, and are \emph{not} needed for the proof of Theorem \ref{thm:gap-stability}.

\begin{lemma}
\label{lem:weighted-projective-lines}
Let $M$ be a $3$-dimensional $\F_q$-space, let
$\mathcal L\subseteq\Gr_2(M)$, and let
$z=\sum_{L\in\mathcal L}a_L[L]$ be a nonzero $2$-chain in $M$.
If every $a_L\ne0$ and $\sum_{L\in\mathcal L}a_L=0$, then
\[
 |\partial z|\le q|\mathcal L|.
\]
\end{lemma}

\begin{proof}
View the $1$-spaces of $M$ as projective points and the members of
$\mathcal L$ as projective lines. If the latter are not all
concurrent, choose three forming a projective triangle. Their union
has $3q$ points, and every further line adds at most $q$ new points.
Hence the union has at most $q|\mathcal L|$ points.

If all members of $\mathcal L$ pass through one point, their union
has $1+q|\mathcal L|$ points. The coefficient of the common point in
$\partial z$ is $\sum_La_L=0$, so $\partial z$ is supported on at
most the remaining $q|\mathcal L|$ points.
\end{proof}

\begin{proposition}[Expansion in a $3$-space]
\label{prop:three-dimensional-expansion}
If $\dim V=3$, then
\[
 h_{0,3}=[3]_q,
 \qquad
 h_{1,3}=\frac1q,
 \qquad
 h_{2,3}=1.
\]
\end{proposition}

\begin{proof}
Since $C^0(V)=K\delta_{\langle0\rangle}$ and
$d_0\delta_{\langle0\rangle}$ is supported on every $1$-space of
$V$, we have $h_{0,3}=[3]_q$.

For a cochain $\varphi=\sum_Ua_U\delta_U$, write
$\widehat\varphi=\sum_Ua_U[U]$ for the chain with the same
coordinates. Fix $\alpha\in C^1(V)$ and put $\beta=d_1\alpha$. Each
$1$-space of $V$ has $q+1$ $2$-space cofaces, so the sum of the
coefficients of $\beta$ is zero. Lemma~\ref{lem:incidence-commutator},
in degree one and ambient dimension three, gives
\[
 \widehat\alpha-d_0^V\bigl(\partial_1\widehat\alpha\bigr)
 =-\partial_2\widehat\beta.
\]
If $\beta=0$, this identity shows that $\alpha\in B^1(V)$. Otherwise,
$d_0^V(\partial_1\widehat\alpha)$ corresponds to a coboundary, and
Lemma~\ref{lem:weighted-projective-lines} gives
\[
 \operatorname{dist}\bigl(\alpha,B^1(V)\bigr)
 \le\bigl|\partial_2\widehat\beta\bigr|
 \le q|\widehat\beta|
 =q|d_1\alpha|.
\]
Therefore $h_{1,3}\ge1/q$.

For the reverse inequality, choose distinct
$P_1,P_2\in\Gr_2(V)$ and define $\alpha\in C^1(V)$ by
$\widehat\alpha=\partial[P_1]-\partial[P_2]$.
Then $d_1\alpha=\delta_{P_2}-\delta_{P_1}$. 
Moreover,
$|\alpha|=2q$, and subtracting a
constant cochain cannot reduce its support below $2q$: the values
$1$ and $-1$ each occur on $q$ $1$-spaces, while the remaining
$q^2-q+1$ values are zero. When $q$ is odd since $[3]_q-2q=q^2-q+1\ge2q$ the cochain $\alpha$ is minimal in its cohomology class. When $q$ is even, $(q+1)1_K=0$ implies $-1 \neq 1$ in $K$, hence adding a coboundary can only eliminate $q$ spaces of the support of $\alpha$, implying that all other classes are of support size at least $[3]_q-q=q^2+1 \geq 2q$, hence $\alpha$ is minimal. Thus
$\operatorname{dist}(\alpha,B^1(V))=2q$ and $|d_1\alpha|=2$, which
proves $h_{1,3}=1/q$.

Finally, Corollary~\ref{cor:odd-dimensional-exactness} and
finite-dimensional duality give $H^2(V)=0$, and hence
$B^2(V)=\ker d_2$. If $\alpha\notin B^2(V)$, then
$d_2\alpha=\lambda\delta_V$ for some $\lambda\ne0$. For any
$P\in\Gr_2(V)$, the cochain
$\alpha-\lambda\delta_P$ lies in $\ker d_2=B^2(V)$. Hence both
$\operatorname{dist}(\alpha,B^2(V))$ and $|d_2\alpha|$ equal $1$, so
$h_{2,3}=1$.
\end{proof}

\subsection{Proof of the gap-stability theorem}

\begin{corollary}
\label{cor:correcting-codimension-one-chain}
Let $E$ be an $m$-dimensional $\F_q$-space with $m\ge3$. For every
$\beta\in C_{m-1}(E)$, there is $c\in K$ such that
\[
 |\beta-c\,\partial[E]|
 \le
 \frac{2(q+1)^2}{[m-1]_q}|\partial\beta|.
\]
\end{corollary}

\begin{proof}
Fix a nondegenerate symmetric $\F_q$-bilinear form on $E$. Since
$C^0(E)=K\delta_{\langle0\rangle}$ and
$d_0\delta_{\langle0\rangle}=(\partial[E])^\perp$, we have
$B^1(E)=K(\partial[E])^\perp$.
Lemma~\ref{lem:orthogonal-complement-duality} and support preservation
give
$d_1(\beta^\perp)=(\partial\beta)^\perp$ and
$|d_1(\beta^\perp)|=|\partial\beta|$.
Proposition~\ref{prop:cone-size} gives $\kappa_1=2(q+1)$, so
Theorem~\ref{thm:grassmann-coboundary-expansion} gives
$h_{1,m}\ge [m-1]_q/([2]_q\kappa_1)
=[m-1]_q/(2(q+1)^2)$.
Consequently,
\begin{align*}
 \min_{c\in K}|\beta-c\,\partial[E]|
 &=
 \operatorname{dist}\bigl(\beta^\perp,B^1(E)\bigr)\\
 &\le
 h_{1,m}^{-1}|d_1(\beta^\perp)|\\
 &\le
 \frac{2(q+1)^2}{[m-1]_q}|\partial\beta|.
\end{align*}
The inequality involving $h_{1,m}$ is trivial when
$\beta^\perp\in B^1(E)$ and otherwise follows from the definition of
$h_{1,m}$.
\end{proof}

\begin{proof}[Proof of Theorem~\ref{thm:gap-stability}]
Fix $\alpha$ as in Theorem~\ref{thm:gap-stability}, and put
$T=[d+1]_q$, $S=\supp(\alpha)$, and $\delta=|S|-T$.
By Theorem~\ref{thm:minimum-support}, $\delta\ge0$.  If $q\le10$, the hypothesis gives
$|\alpha|<T$, so there is nothing to prove. Hence assume $q>10$. Then
$\delta<(1-10/q)T$ and $|S|<2T$.
Call two $d$-spaces adjacent if their intersection has dimension
$d-1$. Corollary~\ref{cor:quotient-link-defect} suggests that in order to have small defect, the subspaces of $\alpha$ must form a near clique (in the co-dimension $1$ intersection graph, i.e. the Grassmann graph).

The key idea of the proof is to use the known fact on the Grassmann graph - its maximal cliques come in two forms, stars which for $P\in\Gr_{d-1}(V)$ the star centered at $P$ is $\mathcal S_P=\{W\in\Gr_d(V):P<W\}$, and tops, which for $E\in\Gr_{d+1}(V)$ the top of $E$ is $\mathcal T_E=\Gr_d(E)$. The goal is to show that a substantial portion of $S$ is contained in some top.

For $U\in S$, put
$\mathcal M(U)=
\{W\in S\setminus\{U\}:\dim(U\cap W)\le d-2\}$.
Corollary~\ref{cor:quotient-link-defect} gives
$|\mathcal M(U)|\le(1+2/q)\delta$.
We first show that every star is sparse. Fix $P\in\Gr_{d-1}(V)$ and put
$a_P=|S\cap\mathcal S_P|$. Consider the pairs $(U,H)$ such that
$U\in S\cap\mathcal S_P$, $H<U$, and $H\ne P$. Each $U$ contributes
$q[d-1]_q$ such pairs.

For each pair $(U,H)$, the cycle equation at $H$ supplies some
$W\in S\setminus\{U\}$ containing $H$; choose one such $W$ and assign
$(U,H)$ to it. Since $H$ and $P$ are distinct hyperplanes of $U$, we
have $H+P=U$. Consequently, $W\notin\mathcal S_P$, since otherwise
$U=H+P\le W$, forcing $U=W$.

We claim that each fixed $W\in S\setminus\mathcal S_P$ receives at most
$q+1$ pairs. Indeed, if $(U,H)$ is assigned to $W$, then
$H\cap P\le W\cap P$. Moreover, since $H$ and $P$ are distinct
hyperplanes of $U$,
\[
\dim(H\cap P)=d-2.
\]
On the other hand, $P\nleq W$ implies $\dim(W\cap P)\le d-2$. Hence
$H\cap P=W\cap P$. Thus $H$ is a hyperplane of $W$ containing the fixed
codimension-two subspace $W\cap P$. There are exactly $q+1$ such
hyperplanes, and each $H$ uniquely determines $U=H+P$. This proves the
claim.
Double counting the assigned pairs now gives
\[
a_Pq[d-1]_q\le (|S|-a_P)(q+1).
\]
Therefore, using $|S|<2T$ and $d\ge4$, we obtain
\begin{equation}\label{eq:star-sparse}
\begin{aligned}
a_P \le \frac{|S|(q+1)}{q[d-1]_q+q+1}<\frac{2T(q+1)}{q[d-1]_q}<\frac{4T}{q^2}.
\end{aligned}
\end{equation}
Here the last inequality uses $[d-1]_q\ge q^2+q+1$.

We next find a large top. Choose adjacent $U_0,U_1\in S$ which exist by the cycle condition. Put $P=U_0\cap U_1$ and $E=U_0+U_1$.
Observe that every $d$-space adjacent to both $U_0$ and $U_1$ belongs to
$\mathcal T_E\cup\mathcal S_P$. Indeed, if it does not contain $P$,
its intersections with $U_0$ and $U_1$ are distinct hyperplanes that
span it, so it is contained in $E$. It follows that
$S\setminus(\mathcal T_E\cup\mathcal S_P)
\subseteq\mathcal M(U_0)\cup\mathcal M(U_1)$.
Put $t=|S\cap\mathcal T_E|$. The nonneighbor estimate,
\eqref{eq:star-sparse}, $|S|=T+\delta$, and $\delta \leq (1-\frac{10}{q})T$ gives
\begin{equation}\label{eq:large-top}
\begin{aligned}
t
&\ge |S|-2\left(1+\frac2q\right)\delta-a_P\\
&= T+\delta-2\left(1+\frac2q\right)\delta-a_P\\
&>T-\left(1+\frac4q\right)
       \left(1-\frac{10}{q}\right)T-\frac{4T}{q^2}\\
&=\left(\frac6q+\frac{36}{q^2}\right)T
>\frac{6T}{q}.
\end{aligned}
\end{equation}

We now correct inside this top. Write $\alpha=\beta+\gamma$, where
$\beta$ is supported in $\mathcal T_E$ and $\gamma$ outside it, and
put $s=|\gamma|=|S|-t$.
Every $H \in \supp \gamma$ can cancel at most one element of $\partial \beta$ (as otherwise it forces $H \leq E$). Hence
$|\partial\beta|\le s$.

Denoting $\eta=2(q+1)^2/[d]_q$,
since $d\ge4$:
\begin{equation}\label{eq:eta-bound}
\eta
\le\frac{2(q+1)}{q^2+1}
<\frac3q.
\end{equation}
Corollary~\ref{cor:correcting-codimension-one-chain}, applied inside
$E$, gives $c\in K$ such that
$|\beta-c\,\partial[E]|\le\eta s$.

Observe that $c\ne0$. Indeed, since $s<2T$, \eqref{eq:eta-bound} gives
$\eta s<6T/q$, whereas \eqref{eq:large-top} gives $|\beta|=t>6T/q$. Hence as $|\beta-c\,\partial[E]|<\eta s<6T/q$, we have $c\ne0$.

Every member of $\mathcal T_E\setminus\supp(\beta)$ occurs in
$\beta-c\,\partial[E]$. Consequently, $T-t\le\eta s$.
Since $s=|S|-t=\delta+(T-t)$, and since \eqref{eq:eta-bound} and $q>10$ give
$\eta<1$, we obtain $s\le\delta/(1-\eta)$.

Finally, put $\zeta=\alpha-c\,\partial[E]$. This is a cycle. Its parts
inside and outside $E$ have disjoint supports, so
$|\zeta|\le(1+\eta)s\le(1+\eta)\delta/(1-\eta)$.
Recall that by assumption $\delta < (1-\frac{10}{q})T$, and since $\eta < 3/q$ we get
\[
\frac{\delta}{T}
<1-\frac{10}{q}
<1-\frac6q
<\frac{1-\eta}{1+\eta},
\]
Thus $|\zeta|<T$. Theorem~\ref{thm:minimum-support} forces $\zeta=0$, and
hence $\alpha=c\,\partial[E]$.
\end{proof}

\section{Optimal fillings in the middle degree}

Our cone construction stops in the middle degree. Nevertheless, when the ambient space is enlarged by another dimension we show that a natural external filling defined using the coboundary operator is optimal.
\begin{proposition}[Optimal external filling]
\label{prop:optimal-external-filling}
Let $U<V$ have dimensions $2k$ and $2k+1$, respectively. Let
$\alpha\in Z_k(U)$ have minimum support among the chains representing
its nonzero class in $H_k(U)$. Then every
$\beta\in C_{k+1}(V)$ with $\partial\beta=\alpha$ satisfies
\[
 |\beta|\ge q^k|\alpha|.
\]
Equality is attained by
$q^{-k}\bigl(d_k^V-d_k^U\bigr)\alpha$.
\end{proposition}

\begin{proof}
Fix a $1$-space $\ell\le V$ not contained in $U$, and let
$\pi_\ell:V\to U$ be the projection on $U$ with kernel $\ell$. Decompose
$\beta=\beta^{\ni\ell}+\beta^{\not\ni\ell}$, where the first summand
is supported on the $(k+1)$-spaces containing $\ell$ and the second
on those not containing $\ell$. Define
\[
 L_\ell[S]=[\pi_\ell(S)]\in C_k(U)
 \quad(\ell\le S),
 \qquad
 Q_\ell[S]=[\pi_\ell(S)]\in C_{k+1}(U)
 \quad(\ell\nleq S).
\]
Extend both maps linearly on the corresponding coordinate subspaces.
The key to the proof is that applying the projection to the $k$-faces of
$\partial\beta=\alpha$ that do not contain $\ell$ gives the identity:
\[
 \alpha
 =
 q^kL_\ell\bigl(\beta^{\ni\ell}\bigr)
 +
 \partial Q_\ell\bigl(\beta^{\not\ni\ell}\bigr).
\]
Indeed, a $(k+1)$-space containing $\ell$ has exactly $q^k$
hyperplanes not containing $\ell$, all projecting onto the same
$k$-space, while projection commutes with the boundary on spaces
disjoint from $\ell$.

Thus $L_\ell\bigl(\beta^{\ni\ell}\bigr)$ represents $q^{-k}$ times
the class represented by $\alpha$ in $H_k(U)$. Scaling preserves
support-minimality, and projection is injective on the $(k+1)$-spaces
containing $\ell$. Consequently, $|\beta^{\ni\ell}|\ge|\alpha|$.

There are $q^{2k}$ choices of the $1$-space $\ell$ outside $U$, and
each $(k+1)$-space outside $U$ contains $q^k$ such $1$-spaces.
Double-counting pairs $(\ell,W)$ with $\ell\nleq U$ and
$W\in\supp(\beta)$ containing $\ell$ gives
$q^{2k}|\alpha|\le q^k|\beta|$, proving the lower bound.

For the upper bound, Lemma~\ref{lem:incidence-commutator} in the
$(2k+1)$-space $V$ gives
$\partial d_k^V\alpha=q^k\alpha$. In the $2k$-space $U$, the
commutator scalar is $[k]_q-[k]_q=0$, so
$\partial d_k^U\alpha=0$. Hence
$q^{-k}(d_k^V-d_k^U)\alpha$ fills $\alpha$.
Its support consists precisely of the $(k+1)$-spaces outside $U$
that contain a member of $\supp(\alpha)$.
Each member of $\supp(\alpha)$ has exactly $q^k$ such cofaces, and
each such coface has a unique $k$-dimensional intersection with
$U$. Its support therefore has size $q^k|\alpha|$.
\end{proof}

\bibliographystyle{plain}
\bibliography{ref}


\end{document}